\documentclass[11pt,twoside]{article}
\usepackage{a4}
\usepackage{latexsym}
\usepackage{amsfonts}
\usepackage{amsmath}
\usepackage{amssymb}
\usepackage{url}
\usepackage{graphicx}
\usepackage[figuresright]{rotating}
\usepackage[nomarkers,nolists]{endfloat}

\newtheorem{lemma}{Lemma}
\newtheorem{proof}{Proof}
\numberwithin{lemma}{section} \numberwithin{equation}{section}
\newtheorem{example}[lemma]{Example}
\newtheorem{theorem}[lemma]{Theorem}
\newtheorem{corollary}[lemma]{Corollary}
\newtheorem{definition}[lemma]{Definition}

\newtheorem{proposition}[lemma]{Proposition}

\vspace{5cm}

\newcommand{\Res}{\operatorname{Res}}
\newcommand{\erfc}{\operatorname{erfc}}

\begin{document}

  \label{'ubf'}
\setcounter{page}{1}                                 

\markboth {\hspace*{-9mm} \centerline{\footnotesize \sc
   Solving a Bonus-Malus integral equation  }
                 }
                { \centerline                           {\footnotesize \sc
         Kucerovsky and Najafabadi                                              } \hspace*{-9mm}
               }

\vspace*{-2cm}

\begin{center}
{
       {\Large \textbf { \sc  Solving an integral equation arising from the Ruin probability of long-term Bonus--Malus systems
                               }
       }
\medskip

\footnote{We thank NSERC and Shahid Beheshti university for financial support.}
{\sc Dan Z. Ku\v{c}erovsk\'{y}}\\
{\footnotesize Fredericton}\\
{\footnotesize UNB-F
}\\
{\footnotesize e-mail: {\it dkucerov@unb.ca}}
}
\smallskip
{\sc Amir T.Payandeh Najafabadi}\\
{\footnotesize Department of Mathematical Sciences,\\ Shahid Beheshti
University, G.C. Evin,\\ 1983963113, Tehran, Iran.
}\\
{\footnotesize e-mail: {\it amirtpayandeh@sbu.ac.ir}}
\end{center}

\thispagestyle{empty}

\hrulefill

\begin{abstract}{\footnotesize
This article studies in detail the solution of an integral equation due to Rongming. The methods involve complex analysis. As an application, we find the ruin probability of a given Bonus--Malus system in a steady state. We obtain closed form solutions for the ruin probability in certain cases, and we characterize these cases.
We give conditions for the Laplace transform of a ruin probability to extend to a meromorphic function in the complex plane, we prove a very general and  almost sharp  inequality of Lundberg  type, and we extend our results to a doubly stochastic situation.}
\end{abstract}

\hrulefill



\section{Introduction and Motivation}
The integral equation to be solved can be written
\begin{equation}
-(\lambda_1+\lambda_2)\widetilde{\psi}(u)+\lambda_1
E(\widetilde{\psi}(u+C))+\lambda_2\int_0^{u}\widetilde{\psi}(u-x)f_X(x)dx=0.
\end{equation}
where $E(\widetilde{\psi}(u+C)$ is an expectation value and $C$ is a given discrete (or continuous) probability distribution. The above equation is further discussed in Rongming et al. (2007). We give some background and motivation before proceeding.

In insurance, a bonus-malus system (BMS) is a system that adjusts the premium paid by a customer according to his individual claim history.
A bonus usually is a discount in the premium which is given on the renewal of the policy if no claim is made in the previous year. A malus is an increase in the premium if there is a claim in the previous year. Bonus-malus systems are very common in vehicle insurance.

The ruin probability for a given Bonus--Malus system can be studied
via a random premium surplus process. Rongming et al. (2007) considered a surplus
process with random premium and geometric L\'evy investments
return process. They obtained an integral equation for
the ruin probability of such a process. This article focuses on solving that equation, and we provide a simple derivation of the equation in the Appendix. We should also mention some other authours who have provided
some functional equation for the ruin probability of such surplus
processes, \textit{e.g.}, Lappo (2004) derived an equation for the
ruin probability of a surplus process which contains two compound
Poisson processes. . Lim \& Qi (2009)
considered a discrete-time surplus process and established an
equation for the ultimate ruin probability. Moreover, they
obtained an upper bound for the ruin probability and studied its
properties via a simulation study.

In the context of Bonus--Malus systems, Wuyuan \& Kun (2005)
considered the Bonus--Malus system's surplus to be two compound
Poisson processes. Then, they proved the Lundberg  inequality for
the ruin probability of such a surplus process.

We consider here the steady state Bonus-Malus problem, similar to that studied by Rongming et al. (2007). We obtain very detailed information about the ruin probability of the model studied by Rongming et al.

Consider a K levels Bonus-–Malus system, with transition
probability matrix $P,$ which stabilizes, in the long run, around an
equilibrium distribution ${\pi}.$ This assumption can be
justified by considering that the annual claim numbers have been
assumed to be i.i.d and each individual policyholder will
ultimately stabilize around an equilibrium level (Denuit, et al.,
2007, \S 4.4). Indeed, in the analysis of the long-term behaviour
of a Bonus--Malus system it is not required to study all step transition
probability matrices and one may use the steady--state distribution
rather than the transition probability matrix $P,$ since each
individual policyholder with be stabilized around one of the level
equilibrium distribution ${\pi}=(\pi_1,\cdots,\pi_K)^\prime.$
Therefore, we may assume the random premium of such a long-term
Bonus--Malus system has been distributed with probability mass
function ${\bf\pi}=(\pi_1,\cdots,\pi_K)^\prime,$ where
$\sum_{k=1}^{K}\pi_k=1.$ Moreover, we suppose that the non-negative and
continuous random claim size X has a density function $f_X$ with
some additional properties to be discussed later (see Definition
1).

The ruin probability for given a Bonus–Malus system can be studied
for {\it either} a short-term {\it or} long-term period of running
time. For the short-term time the random premium of the $(n+1)^{th}$ year
depends on random claims in the $n^{th}$ year. Certainly this assumption
makes the problem very difficult to study in general.

There are two articles by Trufin \& Loisel (2013) and Wua, et at.
(2014) which study ruin probability under the discrete-time risk
model and short-term Bonus--Malus system.

Trufin \& Loisel (2013) assumed the dependence between the random premium
$C$ and the random claim $X$ is a short-term Bonus--Malus system that can
be restated in terms of an exact credibility theorem. Then, they
derived a recursive formula for ruin probability of such
Bonus--Malus system under a discrete-time risk model. Trufin \&
Loisel (2013)'s findings has been generalized to a more suitable
dependent setting,  for just two level Bonus--Malus systems by Wua,
et al. (2014).

Thus, in our setting the surplus process with initial wealth $u$ for the above
Bonus--Malus system is given by
\begin{eqnarray}
\label{surplus-process-BMS}
  U_t &=& u+\sum_{i=1}^{N_1(t)}C_i-\sum_{j=1}^{N_2(t)}X_j\\
\nonumber  &=&u+S_t,
\end{eqnarray}
where  $C_1,C_2,\cdots$ and $X_1,X_2,\cdots,$ respectively, are
two i.i.d. random samples from random premium $C$ and random claim
size $X$ and two independent Poisson processes $N_1(t)$ and
$N_2(t),$ with intensity rates $\lambda_1$ and $\lambda_2,$
stand for claims and purchase request processes, respectively.

The ruin probability for such a Bonus--Malus system, in  a steady state, can be defined by
\begin{eqnarray} \label{Ruin-Definition}
  \psi(u) &=& p(T_{u}<\infty),
\end{eqnarray}
where $u$ is a non-negative real number and $T_{u}$ is the hitting time, i.e., $T_{u}:=\inf\{t:~S_t\leq
-u\}.$

We study the ruin probability through its Laplace transform, extended to the complex plane. We give a theoretical development that allows us in many cases to  determine the  location of the poles in the complex plane and the principal parts of the poles.  In certain cases, that we characterize completely (Theorems \ref{th:characterize.finite.sums} and \ref{linear_systems}),  this means that we are able to give closed form formulas for the ruin probability (see Examples \ref{ex1} and \ref{ex4}.) In other cases, we are usually able to give an infinite series solution for the ruin probability, and when this is not the case, we can  construct approximations (see Theorem \ref{linear_systems} and  Corollary \ref{psi_u}). We demonstrate the effectiveness of our methods by giving a simple method (see Corollary \ref{cor:Lundberg1} and Example \ref{ex0}) to construct Lundberg type inequalities that are more or less the best possible.

We give examples of finding the ruin probability using our method and we consider the case of claim size following
a folded normal distribution, a Chi distribution (generalized Raleigh distribution), and a Gamma distribution. We  generalize our method to study a  doubly stochastic Bonus-Malus problem, which amounts to replacing a discrete distribution by a continuous distribution.

 The rest of this article is organized as
follows.  Section 2 gives general results on the Laplace transform of the ruin probability, including a useful condition for the Laplace transform of the ruin probability to be the restriction of a meromorphic function. Section 2 also gives a method for proving inequalities of Lundberg type.
Section 3 shows that there exists a closed form solution of a certain form for the ruin probability whenever the ruin probability can be assumed to be given by a function of exponential type. Section 3 also considers the existence of solutions by infinite series. Section 4 determines the location of the poles and the principal parts of the poles of the Laplace transform of the ruin probability, and gives our main result on constructively determining the ruin probability, either approximately or in closed form (Theorem \ref{linear_systems}).  Section 5 extends most of our results to a doubly stochastic case. Application of the results for several
Bonus--Malus systems and claim size distributions are
in Section 6. An Appendix collects some mathematical definitions and results that we make use of, proves an existence result for series solutions, and also provides a short proof of a result of Rongming et al (2007) in the case of interest.

\section{Meromorphicity and Lundberg inequalities}

We first use an integral equation provided by Rongming et al. (2007)  to derive properties of the Laplace transform of
 the ruin probability of a given Bonus--Malus
system. 

\begin{theorem}
\label{BMS-Rongming-equation} The Laplace transform of the ruin probability
$\psi(u)$ of the Bonus--Malus system \eqref{surplus-process-BMS} satisfies
\begin{align*}\label{eqn:Rongming.BM}
  \mathcal{L}(\psi(u);u,s) &=  \frac{N_{\psi}(s)}{sD(s)},\\
    \intertext{ where }
D(s)&:=-\lambda_1-\lambda_2+\lambda_1\sum_{k=1}^{K}\pi_ke^{sc_k}+\lambda_2\mathcal{L}(f_X(u),u,s),\\
N_{\psi}(s)&:=D(s)-s\lambda_1\sum_{k=1}^{K}\pi_ke^{sc_k}R_{1-\psi}(c_k,s),\\
R_{1-\psi}(c,s)&:=\int_0^c (1-\psi(u))e^{-su}\,du,
\end{align*}
and $f_X(\cdot)$ is the density of the  random claim probability, $X.$
\end{theorem}
\begin{proof}
Theorem  \ref{Rongming-2007} in the Appendix, with $C$ taken to be the given discrete probability
distribution, shows that the survival probability $\widetilde{\psi}(u),$ of this process  satisfies
\begin{equation}\label{integro-differential-equation}
-(\lambda_1+\lambda_2)\widetilde{\psi}(u)+\lambda_1
E(\widetilde{\psi}(u+C))+\lambda_2\int_0^{u}\widetilde{\psi}(u-x)f_X(x)dx=0.
\end{equation}

 To take the Laplace
transform of this integral equation, we must use  the convolution theorem for Laplace
transforms to simplify the third term,  obtaining
\begin{eqnarray*}
  0 &=& -(\lambda_1+\lambda_2)\mathcal{L}(\widetilde{\psi}(u);s)+\lambda_1\mathcal{L}(\widetilde{\psi}(u);s)\sum_{k=1}^{K}e^{sc_k}\pi_k\\
  &&-\lambda_1\sum_{k=1}^{K}\pi_ke^{sc_k}R_{\widetilde{\psi}}(c_k,s)+\lambda_2\mathcal{L}(\widetilde{\psi}(u);s)\mathcal{L}(f_X(u);u,s),
\end{eqnarray*}
where  $R_{\widetilde{\psi}}(c,s)$ is $\int_0^c (1-\psi(u))e^{-su}\,du.$
We then have
$$\mathcal{L}(\widetilde{\psi}(u);s)= \frac{N_1(s)}{D(s)} $$
where $$N_1(s):=\lambda_1\sum_{k=1}^{K}\pi_ke^{sc_k}R_{1-\psi}(c_k,s)$$ and
$$D(s):=-\lambda_1-\lambda_2+\lambda_1\sum_{k=1}^{K}\pi_ke^{sc_k}+\lambda_2\mathcal{L}(f_X(u),u,s).$$
The Laplace transform of the ruin probability is then given by  $\mathcal{L}({\psi}(u);s)=\frac{1}{s}-\mathcal{L}({\widetilde{\psi}}(u);s)=\frac{1}{s}-\frac{N_1(s)}{D(s)},$
and subtracting these fractions gives the claimed result.
\end{proof}

It is well known that certain aspects of the theory of characteristic functions are advanced by extending a characteristic function to an entire (or meromorphic) function in the complex plane (Lukacs, 1987). In an entirely analogous way, we may regard the Laplace transform of the random premium density function as an entire (or meromorphic) function in the complex plane. We will show that the expression $\frac{N_{\psi}(s)}{sD(s)}$   defined in Theorem \ref{BMS-Rongming-equation} is then a meromorphic function.
We say that a function defined on some part of the complex plane extends to a meromorphic function on the complex plane if there exists a meromorphic function that concides with the given function on the domain of definition of that function. We now give a simple criterion for the Laplace transform of the ruin probability to extend in this way.
We also show that, despite the apparent zero at the origin  in the denominator of $\frac{N_{\psi}(s)}{sD(s)},$ there is in fact no pole at the origin because of a cancellation, and we show that the poles are located in the open left half plane.
\begin{theorem}
\label{negative-Solution} Consider the Bonus--Malus system \eqref{surplus-process-BMS}.
The Laplace transform of the ruin probability extends to a meromorphic function on the complex plane if the moment function associated with the random claim $X$
extends to a meromorphic function on the complex plane.  Moreover, all the poles of $\frac{N_{\psi}(s)}{sD(s)}$ are in the open left half-plane,
and if there is a pole at $-a_k+ib_k$  then there is also a pole at  $-a_k-ib_k.$
\end{theorem}
\begin{proof} By Theorem \ref{BMS-Rongming-equation}, the Laplace transform of the ruin probability is $\frac{N_{\psi}(s)}{sD(s)}$ where $D(s):=-\lambda_1-\lambda_2+\lambda_1\sum_{k=1}^{K}\pi_ke^{sc_k}+\lambda_2\mathcal{L}(f_X(u),u,s),$ and $N_{\psi}(s):=D(s)-s\lambda_1\sum_{k=1}^{K}\pi_ke^{sc_k}R_{1-\psi}(c_k,s).$ Lemma \ref{Laplace-Exponential-Type} in the Appendix uses the Paley-Wiener theorem to show that the functions $R_{1-\psi}(c_k,s)$ are entire. In view of the relationship between complex  moment functions and Laplace transforms, the hypothesis implies that $\mathcal{L}(f_X(u),u,s)$ extends to a meromorphic function. Thus, all the functions appearing in the definition of $N_{\psi}(s)$ and $D(s)$ are, in effect, either entire or meromorphic. Thus, their ratio $\frac{N_{\psi}(s)}{sD(s)}$ is again a meromorphic function. This proves the claim that the Laplace transform of the ruin probability extends to a meromorphic function on the complex plane if the moment function associated with the random claim $X$
extends to a meromorphic function on the complex plane.

We now consider the poles of  $\frac{N_{\psi}(s)}{sD(s)}$.
The terminal value theorem for Laplace transforms (Schiff, 1999, Theorem 2.36) states that if a function $g$ on the real line is bounded and has piecewise continuous derivative,
then its Laplace transform $G$ satisfies $\displaystyle\lim_{s\longrightarrow 0+} sG(s) = \lim_{u\longrightarrow\infty}{g}(u)$ whenever the second limit exists.
The terminal value theorem implies that  $\displaystyle\lim_{s\longrightarrow 0+} \frac{N_{\psi}(s)}{D(s)} = \lim_{u\longrightarrow\infty}{\psi}(u).$ Since ${\psi}$ is a ruin probability, the second limit exists and equals 0.  Since $\frac{N_{\psi}(s)}{D(s)}$ is meromorphic, we can conclude that there is at least a first order zero at the origin, and thus, dividing by $s$, the meromorphic function $\frac{N_{\psi}(s)}{sD(s)}$ is bounded at the origin.

Consider the Laplace integral
 $$g(s)=\int_0^\infty e^{-su}\psi(u)\,du.$$
Just from the fact that $\psi$ is bounded and non-negative, this integral will converge absolutely for each $s>0.$
Recall that $\frac{N_{\psi}(s)}{sD(s)}$ is the analytic continuation of the Laplace transform of $\psi(u)$. Thus, $\frac{N_{\psi}(s)}{sD(s)}$ is equal to the Laplace transform of $\psi(u)$ on at least the positive real line $(0,\infty).$
Since the function $\frac{N_{\psi}(s)}{sD(s)}$ is bounded at $s=0,$ and since $\psi(u)$ is positive or zero everywhere on the positive real line, it follows that the Laplace transform of $\psi$ is a bounded and non-increasing function everywhere on the positive real line. In fact, from the fact that $\psi(u)$ is non-negative, we have $|e^{-(x+iy)u}\psi(u)|=e^{-xu}\psi(u),$ where $x$  and $y$ are real and $u$ is positive. But then, the above Laplace integral converges absolutely for all complex $s$ in the open right half-plane, and there satisfies the inequality $g(x)\geq |g(x+iy)|.$ 

Thus, extending $g(s)$ to a meromorphic function on the complex plane, we see that this function cannot have poles in the open right half plane. Furthermore, a meromorphic function that has a bound everywhere in the open right half plane will not have poles on the imaginary axis.
 We conclude that the meromorphic function $\frac{N_{\psi}(s)}{sD(s)}$  has no poles at any $z\in\mathbb C$ with non-negative real part.

 Since $\psi$ is real-valued,
 it follows that (the analytic continuation of) its Laplace transform is real everywhere on the real line, except at poles,
 and by the Schwartz reflection principle,
 \[\overline{\left(\frac{N_{\psi}(s)}{sD(s)}\right)}=\frac{N_{\psi}(\overline{s})}{\overline{s}D(\overline{s})}.\]
We have thus proven that the poles are located at  $z_j:= -a_j \pm
ib_j,$ with $a_j>0.$
\end{proof}

\begin{proposition} If the moment function associated with $X$ extends to a meromorphic function, and if there is a pole or poles on the real axis, then the rightmost such pole is not to the left of any of the complex poles.\label{prop:real.pole}\end{proposition}
\begin{proof}
Consider the Laplace integral
 $$g(s)=\int_0^\infty e^{-su}\psi(u)\,du.$$
Clearly, for real values of $s$ such that the integral converges, $g(s)$ is a nonincreasing real-valued function. Theorem \ref{negative-Solution} insures that $g(s)$ has a meromorphic extension to the complex plane.
 If we consider how  the nonincreasing function $g(a)=\int_0^\infty e^{-au}\psi(u)\,du$ behaves as $a$ moves to the left along the real axis, we see that either the function increases to infinity, or it remains bounded. If it goes to infinity, then at that point we have a pole. Labelling this pole by $z=-r,$ and supposing that $a$ is a real number that lies to the right of this pole, it follows that $g(a)=\int_0^\infty e^{-au}\psi(u)\,du$ converges absolutely. As in the proof  of  Theorem \ref{negative-Solution}, since  $\psi(u)$ is non-negative, we have $|e^{(a+ib)u}\psi(u)|=e^{au}\psi(u),$ where $x$  and $y$ are real and $u$ is positive. But then,  $g(a)\geq |g(a+ib)|.$ This means that the Laplace transform of the ruin probability cannot have a pole on the vertical line $\Re(z)=a.$ But $a$ was an arbitrary real number greater than  $-r.$  This proves the claim that all the poles of the Laplace transform
lie in the closed left half-plane defined by $\Re(s)\leq -r.$
\end{proof}

We now give a corollary that is very effective at establishing Lundberg-type inequalities, and makes evident a natural lower bound on the exponent that can be used in such an inequality. We should mention that most of the common distributions do have moment functions that extend to meromorphic functions, and that this is a property that is easy to check in specific cases. See Table \ref{table:moment.functions} on page \pageref{table:moment.functions} for a small list of such distributions.

\begin{corollary}\label{cor:Lundberg1} In the  Bonus-Malus problem \eqref{surplus-process-BMS}, if the moment function associated with $X$ extends to a meromorphic function,
and if $-r_0$ lies to the right of any strictly negative real zero of $$D(s):=-\lambda_1-\lambda_2+\lambda_1\sum_{k=1}^{K}\pi_ke^{sc_k}+\lambda_2\mathcal{L}(f_X(u),u,s),$$  then the ruin probability is bounded by $A\exp(-r_0 u).$
\end{corollary}
\begin{proof} Theorem \ref{negative-Solution} shows that the Laplace transform of the ruin probability extends to a meromorphic function. Theorem \ref{BMS-Rongming-equation}   gives an  explicit form for this Laplace transform: $${\mathcal L}(\psi(t),t,s)=\frac{1}{s}-\frac{N_1(s)}{D(s)},$$ where $N_1(s):=\lambda_1\sum_{k=1}^{K}\pi_ke^{sc_k}R_{1-\psi}(c_k,s)$ is entire by Lemma \ref{Laplace-Exponential-Type}. Thus, the only possible poles
of the (meromorphic extension of) the Laplace transform of the ruin probability are those coming from the  zeros of $D(s).$ The function $D(s)$ always has a zero at the origin, but the resulting pole cancels with the $\frac{1}{s}$ term that is displayed above, and thus does not need to be taken into account. If $-r_0$ lies to the right of all negative real roots
 of $D(s)$ then as in the proof of Proposition \ref{prop:real.pole}  the integral  $\int_0^\infty e^{r_0 u}\psi(u)\,du$ converges absolutely. But then the continuous function $e^{r_0 u}\psi(u)$ is in $L^1(\mathbb R^{+}).$ A continuous function that is in $L^1(\mathbb R^{+})$ has an upper bound, $A,$   and evidently $Ae^{-r_0 u}\geq \psi(u).$ This proves the Corollary.
\end{proof}

\begin{example}\label{ex0}
Consider a steady-state Bonus--Malus system with 10 levels, premiums rate ${C}$  and steady-state
distributions, ${ \pi},$ given by $$C=(0.4, 0.8, 1, 1.2, 1.4, 1.5, 1.7, 1.8, 2, 2.2)$$ and   $\pi_n=0.1.$ Take  $X$ to be a folded normal distribution with probability density function $\sim e^{-x^2}.$  We suppose that the number of sold contracts, $N_1(t),$ and the
number of arrived claims, $N_2(t),$ are two independent
processes with intensities $\lambda_1=18$ and $\lambda_2=11.$ We find that the function $D(s)$ for this system is
\begin{multline*}11  \erfc\left( \frac{s}{\sqrt2} \right) e^{\frac{1}{2} s^2}
 +1.8 ( e^{.4  s} +  e^{.8  s}+ e^s  +  e^{1.2  s} +  e^{1.4  s} \\+ e^{1.5  s} +  e^{1.7  s} +e^{1.8 s}   + e^{2 s}+e^{2.2s}) -29.\\ \end{multline*}
As expected, this function is meromorphic (in fact, entire). On the real line, this function $D(s)$ has roots only at $0$ and at $-0.7279947.$ Thus we obtain a bound for the ruin probability of the form $\psi(u)\leq A e^{-0.727994u}.$ We will later re-examine this Bonus-Malus system (see Example \ref{ex2}), and it will then become apparent that we can take $A=\psi(0),$ and that
$\psi(0)$ is approximately $0.6429219.$
\end{example}

\section{Necessary and sufficient conditions for closed form solutions, and existence of infinite series solutions}

In this section we are primarily concerned with existence results. This means that we shall show that there exist solutions of various functional forms, without necessarily being able to explicitly find the coefficients. The question of finding the coefficients is addressed in a later section.

We recall the definition of exponential type $T$ functions on the complex plane.
\begin{definition}
\label{exponential-type} A function $f$ is said to be of exponential type $T$ if there are positive constants $M$ and $T$ such that
$|f(\omega)|\leq M\exp\{T|\omega|\},$ for all $\omega$ in the complex plane.
\end{definition}

\begin{theorem} \label{th:characterize.finite.sums}
Let $\psi(u)$ be the ruin probability of the Bonus-Malus problem \eqref{surplus-process-BMS}. Let us suppose that the moment function associated with the distribution $X$ extends to a meromorphic function on the complex plane.
Then, the following are equivalent:
\begin{enumerate}\item The ruin probability is given by a finite sum of the form
\begin{equation*}
 \psi(u) =  A_0(u) e^{-a_0 u} +  \sum  A_i(u)
e^{-a_i u} \cos b_i u + B_i(u) e^{-a_i u} \sin b_i u,
\end{equation*} where the $A_i(u)$ and $B_i(u)$ are polynomials.
\item the ruin probability is given by an entire function of exponential type.
\end{enumerate}
Both of the above will be true whenever the function $D(s)$ defined in Theorem \ref{BMS-Rongming-equation} has finitely many roots having strictly negative real part.
\end{theorem}
\begin{proof} If condition $2$ holds,  then by Theorems 10.9a/b in Henrici (1991), the Laplace transform of the ruin probability is holomorphic (meaning, analytic and everywhere continuous) outside and on a circle $\Gamma$ about the origin in the complex plane, and the ruin probability can be recovered  from its Laplace transform using Pincherle's inversion formula
\begin{eqnarray}
  \psi(u) &=& \frac{1}{2\pi i} \oint_\Gamma e^{us}
  \frac{N_{\psi}(s)}{sD(s)}ds,\label{eq:pincherle}
\end{eqnarray}
where $\Gamma$ is the abovementioned circle, with the standard counterclockwise orientation.

By Theorem \ref{negative-Solution}, the Laplace transform of the ruin probability, $\frac{N_\psi(u)}{sD(s)},$
  extends to a meromorphic function. Thus, the same is true of the integrand in \eqref{eq:pincherle}, and
a meromorphic function can only have finitely many poles within any compact region of the complex plane. Thus, there are only finitely many poles to consider, and Cauchy's residue theorem then shows that the above integral can be evaluated as a finite sum of the residues at its poles,
\[\psi(u):=\sum \Res\left( e^{us}
  \frac{N_{\psi} (s)}{sD(s)},s=z_k\right) .\]
    Simple poles contribute terms of the general form $\lambda e^{u z_k}$ where $\lambda$ is a complex number. At a pole of  order $m>1$, it follows  from the residue formula (Equation \eqref{eq:residue.formula} on page \pageref{eq:residue.formula}) that we obtain terms of the general form $P(u)e^{u z_k}$ where $P(u)$ is a polynomial of degree $m-1.$
Recalling that the poles can only occur at zeros of $D(s)$ with strictly negative real part,  and that these zeros have the structure described in
 Theorem \ref{negative-Solution} and Proposition \ref{prop:real.pole}, we can rewrite the resulting finite sum in the general form
$ \psi(u) = A_0(u) e^{-a_0 u} +\sum  A_i(u)
e^{-a_i u} \cos b_i u  + B_i(u) e^{-a_i u} \sin b_i u,$ showing that condition $1$ evidently holds.

The converse is straighforward, because we only need to check that a finite sum of the form   $A_0(u) e^{-a_0 u} +\sum  A_i(u)
e^{-a_i u} \cos b_i u  + B_i(u) e^{-a_i u} \sin b_i u,$ is an entire function of exponential type.

If we now assume that the function $D(s)$ has finitely many zeros with strictly negative real part, recall that by Theorem \ref{negative-Solution}, the poles of $\frac{N_\psi(u)}{sD(s)}$ can only occur at zeros of $D(s)$ that have strictly negative real part. Thus, if there are only finitely many such zeros, then there can only be finitely many poles, which are necessarily contained in some sufficiently large circle about the origin. We then proceed as before and obtain a finite sum over residues of $e^{us}\frac{N_\psi(u)}{sD(s)}.$
\end{proof}
It is interesting that the condition given in the above Theorem, namely that the ruin probability should extend to an entire function that grows no faster than some exponential, appears to be mild enough that one could expect this case to arise reasonable frequently, and in fact we had this case occur in two of the examples that appear later on.

It is natural to consider  generalizations to the case of infinite sums. Let us say that a function $f$ is \textit{meromorphic of order zero} if there exist an increasing sequence of contours $\Gamma_n$ that exhaust the plane and are such that the  sequence
$$\oint_{\Gamma_n} \left|\frac{f(z)}{z^{p+1}}\right||dz|$$ 
is bounded. This technical definition insures (Theorem 2.7, Markushevich, 1965) that such functions are determined by the principal parts of their poles (at least up to polynomials of degree $p,$ which can be neglected in our application). For example, the function $\tan z$ is meromorphic of order zero, and writing it in terms of the principal parts of the poles leads to the familiar 
expansion
$$\tan(z) = \sum_{k=0}^{\infty} \frac{-2z}{z^2 - (k + \frac{1}{2})^2\pi^2}.$$

Theorem \ref{Inv-Laplace-th:1} in the Appendix shows that if the Laplace transform of the ruin probability is meromorphic of order zero, then the conclusion of the above Theorem \ref{th:characterize.finite.sums} generalizes to infinite sums. Thus, in this case the probability can be written as an infinite sum of residues, and the convergence is evidently uniform. We now record two Corollaries of Theorem \ref{Inv-Laplace-th:1}:

\begin{corollary}
\label{meromorphic.of.order.zero} Let $\psi(u)$ be the ruin probability of the Bonus-Malus problem \eqref{surplus-process-BMS}. Let us suppose that the Laplace transform of the ruin probability is meromorphic of order zero.
Then,
\begin{eqnarray*}
  \psi(u) &= A_0(u) e^{-a_0 u} +  \displaystyle \sum_{z_i}  A_i(u)
e^{-a_i u} \cos b_i u + B_i(u) e^{-a_i u} \sin b_i u,
\end{eqnarray*} where the $A_i(u)$ and $B_i(u)$ are polynomials, and the $z_j$ are the zeros of $D(s)$ in the open left half-plane.
  \end{corollary}

Residues at simple poles are particularly easy to evaluate, and
the case of simple poles is the generic case.
It thus seems of interest to give a corollary for the case that all the poles are simple poles.

 \begin{corollary}
 \label{psi_u}
Let $\psi(u)$ be the ruin probability of the Bonus-Malus problem \eqref{surplus-process-BMS}. Let us suppose that the Laplace transform of the ruin probability is meromorphic of order zero. The ruin probability $\psi$ has the form
\begin{eqnarray*}
\psi(u) &=& \sum_{\substack{D(z_i)=0\\ \Re(z_i)<0}}
\frac{N_{\psi}(z_i)}{z_iD^\prime(z_i)} e^{z_iu}.
\end{eqnarray*} whenever the derivatives that appear are all non-zero.
\end{corollary}
\begin{proof} Since a zero of $D(s)$ is simple exactly when the derivative $D'(s)$ is non-zero there, the given condition on derivatives insures that all the zeros of $D(s)$ in the open left half plane are simple zeros. The residue of $e^{su}\frac{N_{\psi}(s)}{sD(s)}$ at a simple zero $z_i$ of $D(s)$ is
\[ \lim_{s\longrightarrow z_i} e^{su}\frac{(s-z_i)N_{\psi} (s)}{sD(s)} =e^{z_i u}\frac{N_{\psi}(z_i )}{z_i D^\prime (z_i)}.  \]
\end{proof}
\section{Poles of the Laplace transform}

In this section, we determine the location of the poles and the principal parts of the poles for  the Laplace transform of the ruin probability.

Theorem \ref{BMS-Rongming-equation} gives an equation for the Laplace transform of the ruin probability of the form  \[\mathcal{L}(\psi(u);u,s) =  \frac{N_{\psi}(s)}{sD(s)},\]
where $D(s)$ is known.
This appears to not be useful because it apparently just converts the difficult problem of finding the ruin probability into the equally difficult problem of finding  $N_{\psi}(s).$

However, in most cases, all we need to determine is the behaviour of $N_{\psi}(s)$ at zeros of $D(s)$, and this problem turns out to be tractable.

As was already shown in the proof of Corollary \ref{cor:Lundberg1}, the above equation can be re-written as
 $${\mathcal L}(\psi(t),t,s)=\frac{1}{s}-\frac{N_1(s)}{D(s)},$$
where $N_1(s)$ is entire, so the poles of the Laplace transform can only occur at the origin or at zeros of $D(s).$ Recalling also the restrictions on locations of the poles given by Theorem \ref{negative-Solution}, we  have a fundamental Lemma:
\begin{lemma} The poles of the Laplace transform occur at the (real or complex) zeros of $D(s)$ that are in the open left half-plane.\label{lem:poles.at.zeros.of.Ds}\end{lemma}

Since the function $D(s)$ can be determined in terms of the data given in the problem, we have thus determined the location of the poles. Of course, some of these poles might actually be removable singularities, due to zeros of the numerator. We now consider how to determine the terms of the series solution that (by Corollary \ref{meromorphic.of.order.zero}) we know exists in the meromorphic of order zero case.



In the case of only simple poles, which seems to be the case that occurs in examples, we obtain the following second fundamental result, that allows us to determine any desired number of terms of a series solution. 

\begin{theorem}
\label{linear_systems} Let us suppose all the zeros of $D(s)$ are simple. Let $\{z_1,z_2,z_3,\cdots \, z_K \}$ be the first $K$ zeros of $D(s)$ within the open left half plane, ordered by their modulus. Solve the following system of linear equations for the $A_i.$
\begin{equation}
\label{eq:linearsystem} D'(z_i)A_i =
-\lambda_1 {\displaystyle\sum_\ell \pi_\ell e^{c_\ell z_i} \left\{\frac{1}{z_i}(1-e^{-z_i c_\ell})- A_i c_\ell+\sum_{\substack{j\neq
i,\\j=1}}^{j=K} A_j \frac{1-e^{-(z_i-z_j)c_\ell}}{z_j-z_i}\right\}}{},
\end{equation}
Then,\begin{enumerate}
\item If $D(s)$ has $K$ zeros then the ruin probability is equal to  $\psi(u)=\sum_1^K  A_i e^{z_i u},$ \item If $D(s)$ has more than $K$ zeros then $\sum_1^K A_i e^{z_i u}$ approximates the ruin probability (for large time), and \item If the Laplace transform of the ruin probability is meromorphic of order zero, then the approximations converge uniformly to the true ruin probability as $K$ increases.
\end{enumerate}\end{theorem}
\begin{proof} If there are exactly $K$ roots in total, then Theorem \ref{th:characterize.finite.sums} gives the ruin probability in the general complex exponential form $\psi(u)=\sum_1^K  A_i e^{z_i u}.$  Because the roots are simple, the $A_i$ are scalars rather than polynomials. Theorem \ref{BMS-Rongming-equation} provides the equation \[\mathcal{L}(\psi(u);u,s) =  \frac{N_{\psi}(s)}{sD(s)}.\]
Taking the residue at $s=z_i$ on both sides (using the residue formula, as recalled in Equation \eqref{eq:residue.formula} in the Appendix) we have
$$A_i = \frac{1}{z_i} \frac{N_\psi (z_i)}{D'(z_i)}.$$

Putting in the definition of $N_\psi(u)$ gives
$$A_i = -\lambda_1 \frac{ \sum_k  \pi_k e^{c_k z_i} R_{1-\psi}(c_k,z_i)}{D'(z_i)}.$$
By definition, $R_{1-\psi}(a,s)=\int_0^a (1-\psi(t))e^{-st}\,dt$,
and evaluating this gives
 $$R_{1-\psi}(a,s)=\frac{1}{s}(1-e^{-sa})+ \sum A_i \frac{1}{z_i-s}(1-e^{-(s-z_i)a}).$$
Thus we obtain:
\begin{equation}A_i =
-\lambda_1 \frac{\displaystyle\sum_k \pi_k e^{c_k z_i} \left\{\frac{1}{z_i}(1-e^{-z_i c_k})- A_i c_k+\sum_{j\neq i} A_j \frac{1-e^{-(z_i-z_j)c_k}}{z_j-z_i}\right\}}{D'(z_i)}.
\label{eq:without.moments} \end{equation}
Since $D(s)$, $\lambda_i,$ $\pi_k,$ $c_k,$ and $z_i$ can all be obtained from the data given in
our Bonus--Malus problem, the above is indeed a system of
linear equations for the $A_i.$ Rearranging slightly, we obtain the claimed set of inhomogeneous linear equations.
In the case where the ruin probability is meromorphic of order zero, Corollary \ref{meromorphic.of.order.zero} provides a solution in the form of an infinite sum $\psi(u)=\sum  A_i e^{z_i u},$ and, truncating, we obtain an approximate solution with coefficients determined as above.
If we do not have the meromorphic of order zero condition, then all we can say is that for large time (where the poles of large and negative real part become insignificant) we will have an approximation.
\end{proof}

The above result is phrased in terms of ruin probabilities, but it is readily shown that under the meromorphic of order zero assumption, we can write a similar set of linear equations for the values of the unknown function $N_\psi(s)$ at the roots $z_i$ of $D(s)$. The following corollary gives the principal parts of the poles of the Laplace transform of the ruin probability, in the case where all poles are simple.
\begin{corollary}Let the ruin probability have a Laplace transform that is meromorphic of order zero, and suppose that all the zeros of $D(s)$ are simple. Then, the values of $N_\psi (z_i),$ where the $z_i$ are the roots of $D(s)$, can be determined by solving systems of linear equations (and taking the limit as $K\longrightarrow \infty,$ if needed.) The principal part of the simple pole at $z_i$ is $ \frac{1}{z_i} \frac{N_\psi (z_i)}{D'(z_i)}(s-z_i)^{-1}.$ \end{corollary}

There is no difficulty in principle to generalize Theorem  \ref{linear_systems} and its Corollary to the case of poles of higher order. Each pole of order $m$ contributes a total of $m$ linear equations to a linear system that we solve by standard methods of linear algebra. However, the general result is complicated to state,  and
we have only seen simple zeros in all examples we have looked at. To keep the length of the paper manageable, we have focused on that case.
We now give a brief example of how zeros of higher order can be handled.
\begin{example} Suppose that $D(s)$ has a simple zero at $z_1$ and a double zero at $w_2.$ We then consider the  solution $\psi(u)=A_1e^{z_1u}+A_2e^{w_2u}+B_2ue^{w_2u},$ with the coefficients to be determined from the system of equations
\[\left\{\begin{array}{rl}A_1=&\Res\left(\frac{N_{\psi}(s)}{sD(s)},z_1\right), \\
      A_2=&\Res\left(\frac{N_{\psi}(s)}{sD(s)},w_2\right),\\
                B_2=&\lim_{s\longrightarrow w_2} (s-w_2)^2\left(\frac{N_{\psi}(s)}{sD(s)}\right)\\
\end{array}\right\}.\]
The right hand side of the above will be linear (and inhomogeneous) in the coefficients $A_1,$ $A_2,$ and $B_2. $
\end{example}

The  meromorphic of order zero assumption that is used in the case of a solution by an infinite sum does not need to be checked in practice, because one can verify a formal
solution by substitution into the integral equation of Theorem
\ref{Rongming-2007} in the Appendix. However, it is necessary that the Laplace transform of the given density function be meromorphic or entire, as discussed earlier. The main restriction on the scope of our methods is thus that the Laplace transform of the density function $f_X$ must not have an essential singularity or a  branch point. This easy-to-verify condition is used by us in a  fundamental way. The Cauchy distribution gives an example of a distribution whose density function has an inadmissible Laplace transform.

\section{The doubly stochastic case}
In this section, we show that provided that the random premium density drops off sufficiently quickly (faster than any exponential), the results of the previous section generalize to the following doubly stochastic situation:

Consider the doubly stochastic compound Poisson process
\begin{eqnarray} \label{surplus-process-BMS1}
  U_t &=& u+\sum_{i=1}^{N_1(t)}C_i-\sum_{j=1}^{N_2(t)}X_j\\
\nonumber  &=&u+S_t,
\end{eqnarray}
where  $C_1,C_2,\cdots$ and $X_1,X_2,\cdots,$ respectively, are
two independent i.i.d. random samples from independent random
premium $C$ and random claim size $X$, two independent Poisson
processes $N_1(t)$ and $N_2(t)$ (with intensity rates $\lambda_1$
and $\lambda_2$) respectively, stand for claims and purchase
request processes, and $u$ represents initial wealth $u$ of the
process. Moreover, suppose that the non-negative and continuous random
premium $C$ and claim $X$ respectively have density functions
$f_C$ and $f_X$ that are in Cai's family of distributions.

Rongming et al. (2007) establishes an
integral equation of the form
\begin{equation}\label{integral.equation.DS}
-(\lambda_1+\lambda_2)\widetilde{\psi}(u)+\lambda_1
E(\widetilde{\psi}(u+C))+\lambda_2\int_0^{u}\widetilde{\psi}(u-x)dF_X(x)=0,
\end{equation}
describing the survival probability
$\widetilde{\psi}(u):=1-\psi(u),$ of the surplus process
\eqref{surplus-process-BMS1}. The statement and a short proof are in the Appendix, see Theorem \ref{Rongming-2007} on page \pageref{Rongming-2007}.
The situation studied in the earlier part of the paper can be recovered by taking a discrete and finite distribution for $C.$

Proceeding much as in the proof of Theorem \ref{BMS-Rongming-equation}, we obtain:
\begin{theorem}
\label{New-Rongming-equation.DS} Suppose $U_t$ represents a double
stochastic compound Poisson process of the form given by equation \eqref{surplus-process-BMS1}. Then, the Laplace transform of the ruin probability
$\psi(\cdot)$ of such a doubly stochastic compound Poisson process
satisfies
\begin{align*}\label{eqn:Rongming01}
  \mathcal{L}(\psi(u);u,s) &=
  \frac{N_{\psi}(s)}{sD(s)},\\
\intertext{where}
D(s)&:=-\lambda_1-\lambda_2+\lambda_1M_C(s)+\lambda_2M_X(-s),\\
N_{\psi}(s)&:=D(s)-s\lambda_1\int_{0}^{\infty}e^{sc}R_{1-\psi}(c,s)dF_C(c),\end{align*}
and $M_C(\cdot)$ and $M_X(\cdot),$ respectively, represent the analytic continuations of the
moment generating functions of $C$ and $X.$ The
distribution function of $X$ is denoted $F_C.$
\end{theorem}

We now give a lemma insuring that under sufficiently strong conditions, the integral $\int_{0}^{\infty}e^{sc}R_{1-\psi}(c,s)dF_C(c)$ defines an entire function of $s$.

\begin{lemma} Suppose that the distribution $C$ has a complex moment function that extends to a meromorphic function in the complex plane, and has no poles on the real axis.
Then $\int_{0}^{\infty}e^{sc}R_{1-\psi}(c,s)dF_C(c)$ is entire. \label{lem:Rint.is.entire}
\end{lemma}
\begin{proof}   First we check that the integral $\int_{0}^{\infty}e^{sc}R_{1-\psi}(c,s)dF_C(c)$ will converge. Since $1-\psi(u)$ is a bounded and non-negative function, and $c$ is non-negative and real, we have
\begin{align*} \left| \int_{0}^{\infty}e^{sc}R_{1-\psi}(c,s)dF_C(c)\right|  &\leq
                                                                            \int_0^\infty \left|e^{sc}\right|  \int_0^c (1-\psi(t)) \left| e^{-st}\right| \,dt\,dF_{C} (c)\\
                                                 &= \int_0^\infty e^{\Re(s)c} \int_0^c e^{-t\Re(s)}\,dt\,dF_{C} (c)\\
                                                   &= \int_0^\infty \frac{e^{\Re(s)c}-1}{\Re(s)}\,dF_{C}(c).\\
\end{align*}
This last expression converges for all values of $\Re(s)$.
 Thus the integral $\int_{0}^{\infty}e^{sc}R_{1-\psi}(c,s)dF_C(c)$ converges, and Lemma \ref{Laplace-Exponential-Type} shows that the integrand is an entire function of $s.$
If we restrict the $s$ variable to some arbitrary disk in the complex plane, then we obtain uniform convergence with respect to $s$ in that disk, and this is sufficient to insure that the
integral defines an analytic function that is bounded on  any arbitrary disk in the complex plane --- thus, an entire function.
\end{proof}

For certain distributions, for example, the Cauchy distribution, the complex moment function does not exist.
In most cases, however, the complex moment function does exist, sometimes just in a region of the complex plane, and we may be able to analytically continue the complex moment function to the whole plane.
This assumption holds for most of the common distributions, as can be seen from Table \ref{table:moment.functions} on page  \pageref{table:moment.functions}.

\begin{sidewaystable}
\begin{center}
\begin{tabular}{l c l}
  \hline\large
  Distribution & \large Moment function  &\large  Properties \\
  \hline\normalsize
Bernoulli&  $1-p+pe^t$ & holomorphic\\
Infinite geometric discrete &  $\frac{p}{1-(1-p) e^t},$ $t<-\ln(1-p)$& extends to meromorphic\\

Binomial $B(n, p)$ &  $(1-p+pe^t)^n$& holomorphic for integer $n$ \\

Poisson distribution& $e^{\lambda(e^t-1)}$&meromorphic\\

Uniform distribution $U(a, b)$
 & $\frac{e^{tb} - e^{ta}}{t(b-a)}$& meromorphic\\

Normal distribution $N(\mu, \sigma^2)$
 & $e^{t\mu + \frac{1}{2}\sigma^2t^2}$&holomorphic\\

Chi-squared $\chi^2_k$
& $(1 - 2t)^{-k/2}$&meromorphic for $k$ even\\

Gamma distribution $\Gamma(k, \theta)$
& $(1 - t\theta)^{-k}$&meromorphic for integer $k$\\

Exponential distribution $Exp(\lambda)$
& $(1-t\lambda^{-1})^{-1}$&meromorphic\\

Degenerate point-mass $\delta_a$
& $e^{ta}$&holomorphic\\

Laplace distribution $L(\mu, b)$
& $\frac{e^{t\mu}}{1 - b^2t^2}$&meromorphic\\

Negative Binomial $NB(r, p)$
 &  $\frac{(1-p)^r}{(1-pe^t)^r}$&meromorphic for integer $r$\\
Cauchy $\mbox{Cauchy}(\mu,\theta)$
& does not exist&\\
  \hline
\end{tabular}
\caption{Examples of common moment functions.\label{table:moment.functions}}
\end{center}
\end{sidewaystable}

\begin{theorem}
\label{residue_expansion_DS} Suppose we are given a doubly stochastic system as displayed in
\eqref{surplus-process-BMS1}. Suppose the distribution $C$ has a complex moment function extending to a meromorphic function with no poles on the real line, and suppose
 the random claim density $X$ has a complex moment function that extends to a meromorphic function.
Then,
\begin{eqnarray*}
  \psi(u) &=& \sum_{z_j} \Res \left(e^{su}\frac{N(s)}{sD(s)},z_j\right),
\end{eqnarray*}
where $z_j$ are the zeros of $D(s)$ in the open left half-plane.  \end{theorem}
\begin{proof}  By Theorem \ref{New-Rongming-equation.DS} we have
\begin{eqnarray}\label{eqn:Rongming01}
  \mathcal{L}(\psi(u);u,s) &=&
  \frac{N_{\psi}(s)}{sD(s)},
\end{eqnarray}
where
$D(s):=-\lambda_1-\lambda_2+\lambda_1M_C(s)+\lambda_2M_X(-s),$ and
$N_{\psi}(s):=D(s)-s\lambda_1\int_{0}^{\infty}e^{sc}R_{1-\psi}(c,s)dF_C(c),$
with $M_C(\cdot)$ and $M_X(\cdot),$ respectively, representing the (analytic continuation of the) complex
moment generating functions of $C$ and $X$.
By hypothesis, $M_X$ and $M_C$ are meromorphic. Thus, the function $D(s)$ is meromorphic.
Lemma  \ref{lem:Rint.is.entire} insures that the integral appearing on the right hand side in the above definition of $N_{\psi}(s)$ defines a holomorphic function. Therefore, the meromorphic functions $N_{\psi}(s)$ and $D(s)$ have the same poles, and the same residues at those poles. It follows that at these poles the function $\frac{N_{\psi}(s)}{sD(s)}$ is in fact bounded, and thus the only poles of the function $\frac{N_{\psi}(s)}{sD(s)}$   are the ones coming from zeros of the denominator. The proof of Theorem \ref{negative-Solution} shows there is no pole at the origin, and the proof of Theorem \ref{Inv-Laplace-th:1} gives the claimed expansion.  \end{proof}
The proof of Theorem \ref{linear_systems} can be changed slightly to give the following corollary:
\begin{corollary}
\label{linear_systems.DS} Let us suppose all the zeros of $D(s)$ are simple. Let $\{z_1,z_2,z_3,\cdots \, z_K \}$ be the first $K$ zeros of $D(s)$ within the open left half plane, ordered by their modulus. Solve the following system of linear equations for the $A_i.$
\begin{equation}
\label{eq:linearsystem.DS} D'(z_i)A_i = -\lambda_1\left\{ \frac{M_C(z_i)-1}{z_i}-M'_C (z_i)A_i +\sum_{\substack{j\neq
i,\\j=1}}^{j=K} A_j\frac{M_C(z_i)-M_C(z_j)}{z_j-z_i}\right\},
\end{equation}
\begin{enumerate}\item If $D(s)$ has $K$ zeros then the ruin probability is equal to  $\psi(u)=\sum_1^K  A_i e^{z_i u}.$ \item If $D(s)$ has more than $K$ zeros then $\sum_1^K A_i e^{z_i u}$ approximates the ruin probability. \item If the ruin probability is meromorphic of order zero, then the approximations converge to the true ruin probability as $K$ increases.
\end{enumerate}\end{corollary}
We point out that Corollary \ref{cor:Lundberg1} generalizes without difficulty:
\begin{corollary}\label{cor:Lundberg.DS} In the  doubly stochastic  problem
\eqref{surplus-process-BMS1}, if the moment function associated with $X$ extends to a meromorphic function, the moment function associated with $C$ extends to a meromorphic function that has no poles on the real line, and $-r$ is the first strictly negative real root of the function $D(s):=-\lambda_1-\lambda_2+\lambda_1M_C(s)+\lambda_2M_X(-s),$  then the ruin probability is bounded by $A\exp(-r_0 u),$ where $r_0$ can be chosen as any real number with $-r<-r_0.$
\end{corollary}

\section{Examples and discussion}

We give three examples of solving a Bonus-Malus problem. In one case, there are finitely many roots, so that we may find a basically exact solution in closed form, and in the other two cases the solution is an infinite series. We then give an example in the doubly stochastic case.

\begin{example}\label{ex1} Consider a Bonus--Malus system with 10 levels with premiums rate ${C}$  and steady-state
distributions, ${ \pi},$ given by $$C=(0.4, 0.8, 1, 1.2, 1.4, 1.5, 1.7, 1.8, 2, 2.2)$$ and   $\pi_n=a b^n$ with $b=1.1.$
We suppose that the number of sold contracts, $N_1(t),$ and
number of arrived claims, $N_2(t),$ are two independent
processes with intensity $\lambda_1=18$ and $\lambda_2=11$ and
claim size distributions given as above.

We take the case of $X$ being a gamma distribution with probability density function $Ax^2e^{-6x},$
Finding the real and complex roots of the resulting function $D(s),$ we find that there are only three roots in the open left half plane, $z_1 = -1.53082,$ $z_2= -8.17350-3.76034i,$ and
$z_3= -8.17350+3.76034i.$ Thus we are in the situation described by Theorems \ref{th:characterize.finite.sums} and \ref{linear_systems}, and therefore we  expect a closed form solution. Applying Proposition  \ref{linear_systems} we obtain for the ruin probability (reducing the number of digits for display purposes):
\[\begin{split}\psi(u)&=0.57414\,{{\rm e}^{- 1.5308\,u}} - 0.14781\,\cos \left(  3.7603\,u \right) {{\rm e}^{- 8.1735\,u}}\\
  & - 0.10620\,{{\rm e}^{- 8.1735\,u}}\sin \left(  3.7603\,u \right) \\
 \end{split}\]
When the above is substituted into the given integral equation (see equation \eqref{integro-differential-equation-DS} on page \pageref{integro-differential-equation-DS}) it will satisfy it to a few multiples of machine accuracy. We note that there is no exact formula for the roots of an expression such as $D(s),$ and we thus must approximate the true value of the roots using a numerical method. Increasing the number of digits used when finding the roots increases the absolute accuracy. \end{example}

\begin{example}\label{ex3} We take the case $\pi_i = 0.1$ and $X$ a generalized Raleigh distribution (also known as a Maxwell distribution or a Chi distribution) with probability density function $\sim x^2e^{-x^2}.$  In this case, we find that $D(s)$ has infinitely many zeros in the left half plane, and we looked at the cases of a 1 term series solution, 7 term series solution,  127 term series solution, and a 289 terms series solution (see Figure \ref{fig:chi}a on page \pageref{fig:chi}). We  substituted each solution into the left hand side of the given integral equation (given in the Appendix, as equation \eqref{integro-differential-equation-DS} ), and evaluated the integral. We call the resulting function the error, and the error, for each solution discussed, is shown in Figure  \ref{fig:chi}b.  We note that the function obtained is bounded everywhere, and is largest at the origin. Presumably this is due to having omitted high-order terms of the series solution that decay rapidly and thus contribute primarily close to the origin. We observe the occurrence of a Gibbs-type phenomenon, rather as is seen with partial sums of Fourier series.
\DeclareGraphicsExtensions{.eps}


\end{example}

\begin{example}\label{ex2} We take  the case $\pi_i = 0.1$ and $X$ a folded normal distribution with probability density function $\sim e^{-x^2},$ $C$ being the same as in the previous example. (see Figure \ref{fig:normal}ab on page \pageref{fig:normal}). Comparing this case with the previous case, we see that the ruin probability is initially smaller than in the Raleigh case, but in the Raleigh case the ruin probability decreases sharply until it is smaller than in the case of the folded normal distribution. For larger values of time, the ruin probability for the Raleigh case eventually begins to dominate the ruin probability for the normal case.


\end{example}

Comparing the Raleigh  case with the folded normal case, we see that the ruin probability for the normal case is initially smaller than in the Raleigh case, but in the Raleigh case the ruin probability decreases sharply until it is smaller than in the case of the folded normal distribution. For larger values of time, the ruin probability for the Raleigh case eventually begins again to dominate the ruin probability for the normal case.
We can also say, broadly speaking, in the cases where we have an infinite series solution,   the ruin probability begins by decreasing rapidly, then begins decreasing less rapidly and behaving approximately linearly, and finally it starts to decay exponentially. The  Raleigh case in particular behaves approximately linearly over a wide range.
Thus, we see that it is possible to obtain really detailed information about the relative behavior of the ruin probabilities for different but related Bonus--Malus systems using our method.

Our numerical experience has been that the systems of linear equations obtained are really well-behaved. Possibly this is surprising, because Laplace transform methods tend, broadly speaking, to lead to ill-conditioned numerical problems. The most ill-conditioned system we have seen had condition number 45.8, which occured in the 289 term case, using the Chi distribution.  Thus, we did not see any signs of numerical problems, but we nevertheless used 60 digit floating point arithmetic throughout. The root-finding was based on Newton's method. The roots need to be found accurately, and missing roots can of course perturb the results significantly.

\begin{example}\label{ex4}Consider the doubly stochastic case, taking both $C$ and $X$ to be a gamma distribution, $X\sim \exp(-bu)$
and $C\sim \exp(-au).$ There is only one pole, and we find a closed form solution for the ruin probability, namely
$$\psi(u)= \frac{(a+b)\lambda_2}{b(\lambda_1+\lambda_2)}e^{\frac{-(\lambda_1-\lambda_2)}{(\lambda_1+\lambda_2)}u}$$ where we have supposed that $\lambda_1>\lambda_2.$ We note that the moment functions are just meromorphic, so this is a case where the solution had to be checked by substitution into the integral equation.
\end{example}

 As has been seen, our methods provide an approximation by a sum of real  exponentials, multiplied by trigonometric functions, and sometimes also polynomials. This is a quite general functional form.  The general idea  of approximating an unknown function by sums of small numbers of real exponentials is not new and has been found to be often effective in actuarial applications.
The difference between these \textit{ad hoc} approximations and our above results is that our results are based on advanced mathematical features of the problem that allow us to find precisely those functions that are best suited to the problem. It is likely, thus, that we obtain  the best possible approximations that can be gotten with a small number of terms of the form we are considering.

\section{Conclusion}
We have presented a series solution for the ruin probability of a Bonus-Malus system,  of the form
$$\sum A_i u^{m_i}  e^{-{k_i} u} \sin b_i u + B_i u^{m_i}  e^{-k_i u} \cos b_i u .$$
In certain cases, which we have characterized, the series terminates, thus giving exact solutions in closed form.
In cases where the convergence of the series cannot be determined, we are still able to obtain approximations to the ruin probability. We gave examples showing that the method works.
Clearly, the results presented here  can be employed to: ({\bf 1}) compare two
given Bonus--Malus systems; ({\bf 2}) evaluate behavior of a given
Bonus--Malus system with respect to initial wealth of insurer; and
({\bf 3}) design an optimal Bonus--Malus system, based upon
{either} number of Bonus--Malus system levels {or} tail
behavior of claim size distribution. The second application may be
justified by Xianbin (2005)'s findings who established that for any
Bonus--Malus System, there always exists an unique steady-state
distribution. The third application is comparing some possible
Bonus--Malus systems via their ruin probabilities.
We also mention that, to our knowledge, the problem of precisely when  moment functions
extend to meromorphic functions in the complex plane has never been thoroughly studied, and this may in the future be a problem of interest. \bigskip

\centerline{{\huge {\rotatebox[]{270}{\textdagger}}\!{\rotatebox[]{90}{\textdagger}}}}

\section*{Appendix: Mathematical results, and a proof  of a result of Rongming's}

\begin{definition}
\label{Cai-family-distribution} A continuous random variable $X$ is
a member of  Cai's family of distributions if its density
function satisfies the following conditions:
\begin{description}
    \item[$A_1$)]  $\int_0^\infty |f^\prime(y)|dy<\infty$, $\int_0^\infty |f^{\prime\prime}(y)|dy<\infty$
    \item[$A_2$)] The survival function ${\bar F}(u):=\int_u^\infty f(y)dy$ is twice continuously differentiable on
    $[0,\infty);$ and both ${\bar F}^\prime(u)$ and ${\bar F}^{\prime\prime}(u)$ are
    bounded on $[0,\infty).$
\end{description}
\end{definition}
Many well known claim size distributions, such as exponential, exponential mixtures, Erlang, Pareto (with finite mean), Lognormal,
etc, satisfy the above conditions, see Cai (2004) for more detail.

The well known Paley-Wiener theorem states that the
Fourier transform of an $L_2({\Bbb R})$ function vanishes outside
of an interval $[-T,T]$ if and only if the given function extends to an entire
function of exponential type $T$, see Dym \& McKean (1972, page
158) for more detail. \label{Paley.Wiener}

We also have need of the theory of
residues for meromorphic functions. Briefly, meromorphic functions are analytic functions without essential singularities or branch points, and $\Res(f,z_i)$ is the
coefficient of $\frac{1}{z-z_i}$ in the Laurent series expansion
of $f$ around $z_i.$   The importance of the
theory of residues perhaps rests on two facts: many integrals can
be evaluated in terms of a sum of residues, and residues can be evaluated efficiently by formulas:
\begin{equation}\Res(f,z_0)=\frac{1}{(m-1)!}\lim_{s\rightarrow z_0}\left(\frac{d}{ds}\right)^{m-1}\!\!\!\!\!\!\!\!\!(s-z_0)^m f(s)
\label{eq:residue.formula}\end{equation}
where $m$ is the order of the pole, see for example Ablowitz \&
Fokas (1990, \S 4). Poles of order $1$ are also known as simple poles.

\begin{lemma}\label{Laplace-Exponential-Type}
Suppose $f(x)$ is bounded and real on the real axis. Then, $R_f(a,s):=
\mathcal{L}(f(t),t,s)-e^{-as}\mathcal{L}(f(t+a),t,s),$ as a function of $s,$ is in the
range of the Laplace operator, is entire, and is of exponential type.
\end{lemma}
\begin{proof} By the translation property of the Laplace
transform, $e^{-as}\mathcal{L}(f(t+a);t,s)$ is equivalent to
$\mathcal{L}(u_a (t) f(t);t,s)$ where $u_a$ is the Heaviside step
function, defined by
\[ u_a (t):= \begin{cases}0 &\text{if $t\leq0$, and}\\1 &\text{if $t>0$.} \end{cases}\]
Thus, $R_f(a,s)=\mathcal{L}(\{f(t)-u_a(t)f(t)\};t,s)$ is
in the range of the Laplace operator, as claimed. Moreover, it is
the Fourier-Laplace transform of a compactly supported function,
and thus by the aforementioned Paley-Wiener theorem  for compactly
supported functions, it is analytic and of
exponential type, with respect to the $s$ variable.
\end{proof}
Let us say that a function $f$ is \textit{meromorphic of order zero} if there exist an increasing sequence of contours $\Gamma_n$ that exhaust the plane and are such that the  sequence
$$\oint_{\Gamma_n} \left|\frac{f(z)}{z^{p+1}}\right||dz|$$ is bounded. (By the term exhausting the plane, we mean that one contour lies successively inside the other and that the distance from $\Gamma_n$ to the origin tends to infinity. See page 56 of Markushevich (1965). )

\begin{theorem}
\label{Inv-Laplace-th:1}\label{existence.thm.II} Suppose that a given Bonus--Malus system
satisfies the conditions given by Theorem \ref{BMS-Rongming-equation}, and that the Laplace transform of the ruin probability is meromorphic of order zero.
Then,
\begin{eqnarray*}
  \psi(u) &=& \sum_{z_j} \Res \left(e^{su}\frac{N_\psi(s)}{sD(s)},s=z_j\right),
\end{eqnarray*}
where the $z_j$ are the zeros of $D(s)$ in the open left half-plane.
  \end{theorem}
\begin{proof}  Lemma \ref{lem:poles.at.zeros.of.Ds} insures that the poles of the
Laplace transform of $\psi(u)$ are all at zeros of $D(s)$ that are in the open left half plane.
Thus, the inverse Laplace transform is given by Mellin's inverse formula, in terms of the following Cauchy-type integral:
\begin{eqnarray}
  \psi(u) &=& \frac{1}{2\pi i} \int_\Gamma e^{us}
  \frac{N_{\psi}(s)}{sD(s)}ds.\label{eq:inverse.Laplace.int}
\end{eqnarray}
where $\Gamma$ is the  vertical line $\{z|\, \Im(z)=0\}$ in the
complex plane ${\Bbb C}$, bypassing the origin by a semi-circular cutout to the left.
From the meromorphic  of order zero condition  we have
that the integrals \[\oint_{\Gamma_n} \left|{ \frac{N_{\psi}(s)}{s^{p+2}D(s)}}\right||dz|\]
 are bounded. Along a ray $s:=re^{i\theta}$ that happens to be in the left half plane, the function $e^{su},$ where $u$ is positive, will decay faster than any power of $\frac1s .$ If we denote by $\Gamma'_n$ the  part of the contours $\Gamma_n$ that lie in the left half plane, we see that  the integrals
$\int_{\Gamma'_n} e^{us}
  \frac{N_{\psi}(s)}{sD(s)}\,ds$
    will tend to zero as $n$ goes to infinity. Using the curves $\Gamma'_n$ to close the contour in the integral \eqref{eq:inverse.Laplace.int}, Cauchy's residue theorem gives, for all positive $u,$
\[\displaystyle \psi(u) = \frac{1}{2\pi i} \int_\Gamma e^{us}
  \frac{N_{\psi}(s)}{sD(s)}ds = \sum_{z_j} \Res \left(e^{s u}\frac{N_\psi(z_j)}{sD(s)},s=z_j\right).\]   \end{proof}

The next result is due to Rongming et al. (2007), we give a simpler proof of their result.
Consider a double stochastic compound Poisson process
\begin{eqnarray} \label{surplus-process-DS}
  U_t &=& u+\sum_{i=1}^{N_1(t)}C_i-\sum_{j=1}^{N_2(t)}X_j\\
\nonumber  &=&u+S_t,
\end{eqnarray}
where  $C_1,C_2,\cdots$ and $X_1,X_2,\cdots,$ respectively, are
two i.i.d. random samples from independent random premium $C$ and
random claim size $X$, two independent Poisson processes $N_1(t)$
and $N_2(t)$ (with intensity rates $\lambda_1$ and $\lambda_2$)
are, respectively, stand for claims and purchase request
processes, and $u$ represents initial wealth $u$ of the process.
\begin{theorem}\label{Rongming-2007}
The survival probability $\widetilde{\psi}(u)$ of the double
stochastic compound Poisson process \eqref{surplus-process-DS}
satisfies
\begin{equation}\label{integro-differential-equation-DS}
-(\lambda_1+\lambda_2)\widetilde{\psi}(u)+\lambda_1
E(\widetilde{\psi}(u+C))+\lambda_2\int_0^{u}\widetilde{\psi}(u-x)dF_X(x)=0,
\end{equation}
where $\lim_{u\rightarrow\infty}\widetilde{\psi}(u)=1,$ $F_X$
stands for the common distribution function of random claim size
$X$ and $E(\widetilde{\psi}(u+C))$ represents an expectation under
random premium $C.$
\end{theorem}
\emph{Proof.} Considering the surplus process $U_t$ in a very
short time interval $(0,\Delta t],$ the following five possible
events can be considered:
\begin{description}
    \item[(1)] $A_1\equiv(\hbox{ no claim and no premium income in }(0,\Delta
    t]),$ where $P(A_1)=(1-\lambda_1\Delta t)(1-\lambda_2\Delta t)+o(\Delta t)$
    \item[(2)] $A_2\equiv(\hbox{no claim and one premium income in }(0,\Delta
    t]),$ where $P(A_2)= \lambda_1\Delta t(1-\lambda_2\Delta t)+o(\Delta t)$
    \item[(3)] $A_3\equiv(\hbox{one claim and no premium income in }(0,\Delta
    t]),$ where $P(A_3)= \lambda_2\Delta t)(1-\lambda_1\Delta t)+o(\Delta t)$
    \item[(4)] $A_4\equiv(\hbox{one claim and one premium income in }(0,\Delta
    t]),$ where $P(A_4)= \lambda_1\Delta t\lambda_2\Delta t+o(\Delta t)$
    \item[(5)] $A_5\equiv(\hbox{other cases occurring in }(0,\Delta
    t]),$ where $P(A_5)=o(\Delta t).$
\end{description}
By conditioning of the survival probability $\widetilde{\psi}(u)$
on the above five cases, one may obtain
\begin{eqnarray*}
  \widetilde{\psi}(u) &=&
  P(A_1)\widetilde{\psi}(u)+P(A_2)\int_{0}^{\infty}\widetilde{\psi}(u+c)dF_C(c)+
    \\&&P(A_3)\left[\int_{0}^{u}\widetilde{\psi}(u-x)dF_X(x)+\int_{u}^{\infty}0dF_X(x)\right]
    \\&&+P(A_4)\int_{0}^{\infty}\int_{0}^{\infty}\widetilde{\psi}(u+c-x)dF_C(c)dF_X(x)+o(\Delta
  t)\\
  &=&\left [1-(\lambda_1+\lambda_2)\Delta t+\lambda_1\lambda_2(\Delta t)^2
  \right]\widetilde{\psi}(u)+
    \\&&\left[\lambda_1\Delta t-\lambda_1\lambda_2(\Delta t)^2
  \right]\int_{0}^{\infty}\widetilde{\psi}(u+c)dF_C(c)\\
  &&+\left[\lambda_2\Delta t-\lambda_1\lambda_2(\Delta t)^2
  \right]\int_{0}^{u}\widetilde{\psi}(u-x)dF_X(x)\\
  &&+\lambda_1\lambda_2(\Delta t)^2\int_{0}^{\infty}\int_{0}^{\infty}\widetilde{\psi}(u+c-x)dF_C(c)dF_X(x)+o(\Delta
  t)\\
  \Rightarrow && \left [-(\lambda_1+\lambda_2)\Delta t+\lambda_1\lambda_2(\Delta t)^2
  \right]\widetilde{\psi}(u)+\\
    &&\left[\lambda_1\Delta t-\lambda_1\lambda_2(\Delta t)^2
  \right]\int_{0}^{\infty}\widetilde{\psi}(u+c)dF_C(c)\\
  &&+\left[\lambda_2\Delta t-\lambda_1\lambda_2(\Delta t)^2
  \right]\int_{0}^{u}\widetilde{\psi}(u-x)dF_X(x)\\
  &&+\lambda_1\lambda_2(\Delta t)^2\int_{0}^{\infty}\int_{0}^{\infty}\widetilde{\psi}(u+c-x)dF_C(c)dF_X(x)+o(\Delta
  t)
\end{eqnarray*}
Dividing both sides by $\Delta t$ and letting $\Delta t \longrightarrow 0+$ leads to the desired result. $\square$

\label{'ubl'}
\end{document}